\documentclass[12pt,oneside,reqno]{amsart}
\usepackage[utf8]{inputenc}
\usepackage[T1]{fontenc}
\usepackage[margin=0.9in]{geometry}
\usepackage{
amsmath,
amssymb,
amsthm,
amsrefs,
mathrsfs,
dsfont, 
xcolor,
graphicx,
float,
enumitem, 
comment,
caption, 
soul 
}
\usepackage{amsrefs}
\usepackage{mathrsfs}
\definecolor{mygray}{gray}{0.3} 
\usepackage[colorlinks=true,urlcolor=gray,linkcolor=mygray,citecolor=mygray]{hyperref}
\usepackage{graphicx}
\usepackage{fancyhdr} 
\newtheorem{theorem}{Theorem}
\newtheorem{proposition}[theorem]{Proposition}
\newtheorem{lemma}[theorem]{Lemma}

\newtheorem{claim}{Claim}

\newtheorem*{rmk}{Remark}

\newtheorem{proofc}{Proof of Claim}

\newcommand{\de}{\delta}

\newcommand{\eps}{\epsilon}

\newcommand{\la}{\lambda}

\renewcommand{\phi}{\varphi}

\newcommand{\N}{\mathds{N}}

\newcommand{\R}{\mathds{R}}
\newcommand{\C}{\mathds{C}}
\newcommand{\D}{\mathds{D}}
\newcommand{\PP}{\mathbb{P}}

\newcommand{\E}{\mathbb{E}}

\newcommand{\CUE}{\mathrm{CUE}}

\newcommand{\Sp}{\mathrm{Sp}}
\newcommand{\tr}{\mathrm{tr}}
\newcommand{\lv}{\lvert}
\newcommand{\re}{\mathrm{e}}
\newcommand{\ri}{\mathrm{i}}

\newcommand{\diag}{\mathrm{diag}}

\title[Zeros Of Random Analytic Functions And Spectral Properties Of Perturbed Unitary Matrices ]{Zeros Of Random Analytic Functions And Spectral \\ 
Properties Of Perturbed Unitary Matrices}
\author{Aniss Fares}
\address{CMLS (École Polytechnique), 91120 Palaiseau, France}
\email{aniss.fares@polytechnique.edu}
\begin{document}

\begin{abstract}
    We study the spectral properties of a rank-one multiplicative perturbation of a unitary matrix, a model introduced by Fyodorov. Building upon earlier results by Forrester and Ipsen, we provide a direct proof that the eigenvalues converge to the zeros of a specific Gaussian analytic function. Our approach extends these results to other unitarily invariant models. This method enables us to address a question raised by Dubach and Reker concerning the critical timescale at which an outlier emerges.
\end{abstract}
\maketitle

\section*{Introduction}

The study of the interplay between random matrices and the zeros of random analytic functions has revealed deep and surprising connections, as beautifully illustrated in the monograph \cite{GAF}. In a variety of contexts, the spectral behavior of random matrices has been shown to exhibit striking links with the zeros set of certain Gaussian analytic functions, suggesting a unifying structure behind seemingly distinct probabilistic objects. \\
\\
A recent and particularly elegant manifestation of this phenomenon is due to Forrester and Ipsen \cite{ForIps}, who considered a multiplicatively perturbed Haar matrix, a model originally introduced by Fyodorov in \cite{Fyo}. They showed that, in a critical scaling regime, the eigenvalues of this model converge to the zeros of a specific Gaussian analytic function. \\
\\ 
Interestingly, the same Gaussian analytic function (with a change of variable $z\mapsto 1/z$) also arises in the context of rank-one additive perturbations of Hermitian models as shown in \cite{Tao}. The statistical properties of this function have been the subject of further investigation in \cite{FKP}.\\
\\
The model considered in \cite{ForIps} has its roots in physics, and more specifically in scattering theory, where similar structures appear in the study of open quantum systems. A more detailed exposition of this physical background can be found in \cite{Sav}.\\ 
\\
More recently, this same model has been explored from a dynamical point of view by Dubach and Reker \cite{DubRek}, who investigated the emergence of an outlier.\ They showed that a "strongly separated outlier" can be observed up to a critical time, and that this time is optimal in the case of Haar matrices. \\
\\
The present work serves two complementary purposes. First, we revisit the analysis of Forrester and Ipsen, offering an alternative approach that naturally extends to a broader class of unitary matrices. This new perspective allows us to generalize their convergence results beyond the Haar case. Furthermore, we obtain an extension of the critical time optimality result of Dubach and Reker to non-Haar settings. While their proof relied heavily on the integrability of the $\CUE$, our approach reveals a structural robustness of the phenomenon, valid for a much wider class of random unitary matrices. 

\subsection*{Definitions and main results.}
Throughout the paper $\D$ denotes the open unit disk. Let $H(\D)$ denote the set of analytic functions on $\D$. It is a Polish space when endowed with the topology of uniform convergence on compact subsets. We equip this topological space with its Borel $\sigma$-algebra. We define a \textit{random analytic function} as a random variable taking values in $H(\D)$.\\

A positive Borel measure on $\D$ is said to be a \textit{Radon measure} if it assigns a finite mass to every compact set $K$ in $\D$. For an analytic function $f$ we denote by $\mathcal{Z}_f$ its zero set. For an analytic function 
$f$ that is
not identically zero, the \textit{empirical measure of}$f$ is defined as 
\begin{equation*}
    \sum_{w\in \mathcal{Z}_f}\de_w,
\end{equation*}
where the zeros are repeated with multiplicities. The identity theorem implies that an analytic function that is not identically zero must have a finite number of zeros in every compact $K$ of $\D$. Thus, its empirical measure is a Radon measure.\\
Let  $\mathcal{Q}$ denote the space of integer-valued Radon measures. Endow $\mathcal{Q}$ with the smallest topology that makes the maps 

\begin{equation*}
    \mu \in \mathcal{Q} \mapsto \int_{\D}\phi(z)d\mu(z):=\mu(\phi)
\end{equation*}
continuous, for all continuous compactly supported functions $\phi$ with support in $\D$. With this topology $\mathcal{Q}$ is a Polish space (see \cite{Olav}, chapter 4). A random variable taking values in $\mathcal{Q}$ is called a \textit{point process}. If a random analytic function is almost surely not identically zero, then its empirical measure is a point process.\\

We say that a sequence of point processes $\{\mu_N\}$ \textit{converges vaguely} to a point process $\mu$ if 
\begin{equation*}
    \mu_N(\phi) \xrightarrow[N \to \infty]{d}\mu(\phi)
\end{equation*}
for all $\phi \in \mathscr{C}_c(\D) $. We denote this convergence by $\mu_N \xrightarrow[N \to \infty]{v} \mu$. The topology defined on $\mathcal{Q}$ above coincides with the topology associated with vague convergence, and we refer to it as the vague topology (see \cite{Olav}, chapter 4). \\

We can now state the following result of Forrester--Ipsen in \cite{ForIps}, for which we provide a new proof. 
\begin{theorem}\label{Forrester-Ipsen result}
    Let $\{c_k\}$ be a sequence of complex independent standard Gaussian random variables.
    Let $U$ be an $N \times N$ Haar-distributed matrix and $A=\diag(aN^{-1/2},1,\ldots,1)$ with $a \in \C^{*}$. Then 
    \begin{equation*}
        \sum_{\la \in \Sp(UA)}\de_{\la}\xrightarrow[N \to \infty]{v}\sum_{w\in \mathcal{Z}_{\phi_a}}\de_w
    \end{equation*}
    with 
    \begin{equation*}
        \phi_a(z):= a - \sum_{k \geq 1}c_k z^k, \ \ z\in \D.
    \end{equation*}
    That is, the point process formed by the eigenvalues of $UA$ converges vaguely to the point process formed by the zeros of $\phi_a$.
\end{theorem}
In \cite{ForIps} the authors characterized the eigenvalues of $UA$ as the zeros of a certain random analytic function converging to $\phi_{a'}$; for the required tightness they relied on a previous result from Krishnapur in \cite{Krishnapur}. We adopt a similar strategy but perform the calculation differently, obtaining a random analytic function for which tightness can be directly established. This new approach enables us to generalize Theorem \ref{Forrester-Ipsen result} in two directions. \\
 First, we retain the assumption that $U$ is Haar-distributed, but we allow for more general rank-one perturbations that are not necessarily aligned with a fixed direction. More precisely we now consider $A=I_N - (1-aN^{-1/2})vv^*$ where $v$ is a unit vector that can even be random as long as it is independent from $U$ and distributed with a non-vanishing density on $S_{N-1}(\C)$. Note that when $v=e_1$ then $A=\diag(aN^{-1/2},1,\ldots,1)$ as in Theorem \ref{Forrester-Ipsen result}. Our first main result establishes that Theorem \ref{Forrester-Ipsen result} still holds in this broader setting.
\begin{theorem}\label{Extension result}
   Let $\{c_k\}$ be a sequence of complex independent standard Gaussian random variables.
    Let $U$ be an $N \times N$ Haar-distributed matrix and $A=I_N-(1-aN^{-1/2})vv^*$ with $a \in \C^{*}$ and $v$ a unit vector as previously defined. Then 
    \begin{equation*}
        \sum_{\la \in \Sp(UA)}\de_{\la}\xrightarrow[N \to \infty]{v}\sum_{w\in \mathcal{Z}_{\phi_a}}\de_w
    \end{equation*}
    with 
    \begin{equation*}
        \phi_a(z):= a - \sum_{k \geq 1}c_k z^k, \ \ z\in \D.
    \end{equation*}
    That is, the point process formed by the eigenvalues of $UA$ converges vaguely to the point process formed by the zeros of $\phi_a$.
\end{theorem}
The second way in which we generalize Theorem \ref{Forrester-Ipsen result} is by considering the model $M=VDV^*$, where $V$ is Haar-distributed, and $D$ is a random diagonal matrix with i.i.d entries, independent of $V$. This is a natural generalization of the decomposition $U=VDV^*$ which is satisfied almost surely by any Haar-distributed matrix $U$, where $\PP(D)\sim \prod _{k<l}\lvert \re^{i\theta_k}-\re^{i\theta_l}\lvert$. The main difference between the two models is that the eigenvalues of $M$ are independent and no longer repel each other. We will furthermore assume that $D$ almost surely has a simple spectrum  and that, if we write  $D=\diag(Z_1,\ldots,Z_N)$ then the sequence of moments of $Z_1$ decays
faster than any polynomial, that is, $\sup_{k\geq 1}\lvert \E[Z_1^k]\lvert=o(N^{-b})$ for any integer $b$. Our second main result shows the conclusions of Theorem \ref{Forrester-Ipsen result} still hold in this more general setting.
\begin{theorem}\label{extension bis}
    Let $\{c_k\}$ be a sequence of complex independent standard Gaussian random variables.
    Let $M=VDV^{*}$ with $V$ and $D$ as previously defined.  We also define the matrix $A:=I_N-(1-aN^{-1/2})vv^*$ with $a \in \C^{*}$ and $v$ a unit vector. Then
    \begin{equation*}
        \sum_{\la \in \Sp(MA)}\de_{\la} \xrightarrow[N \to \infty]{v}\sum_{w \in \mathcal{Z}_{\phi_{a'}}} \de_w
    \end{equation*}
    with
    \begin{equation*}
        \phi_{a'}(z)= a'- \sum_{k\geq1}c_kz^k, \ z\in \D, \ a':=\frac{a}{\sqrt{2}}.
    \end{equation*}
    That is, the point process formed by the eigenvalues of $MA$ converges vaguely to the point process formed by the zeros of $\phi_{a'}$.
\end{theorem}
Note that this function differs from the one in Theorem \ref{Forrester-Ipsen result} only by its constant coefficient.\\

The condition $\sup_{k \geq 1}\lvert \E[Z^k]\lvert=o(N^{-b})$ for all integers $b$ is satisfied by a broad class of random variables including : 
\begin{itemize}
    \item  The uniform law on the unit circle since $\E[Z^k]=\de_{k0}$.
    \item  The wrapped Normal distribution, that is, $Z=\re^{\ri X}$ with $X \sim \mathcal{N}(\mu, \sigma^2)$. In this case $\E[Z^{k}]=\re^{\ri k \mu -k^2\sigma^2/2}$ and $\sup_{k \geq 1}\lvert \E [Z^k]\lvert = \re^{-\sigma^2/2}$. Therefore taking, for instance, $\sigma^2=N^m$ with $m> 0$ or $\ln(N)$ satisfies the condition.
    
    \item The wrapped Cauchy distribution, that is, $Z=\re^{\ri X}$ with $X$ following a Cauchy law with parameter $x_0 \in \R$ and $\gamma >0$. In this case $\E[Z^{k}]=\re ^{-\gamma \lvert k \lvert + \ri k x_0}$ and $\sup_{k \geq 1}\lvert \E [Z^k]\lvert = \re^{-\gamma}$. Therefore taking, for instance, $\gamma$ as any positive power of $N$ satisfies the condition.
\end{itemize}
For a proof of the above results and further background on these random variables and their relevance, notably in statistics, see \cite{Kanti}.\\
\\
In \cite{DubRek} the authors studied a dynamical version of the $UA$ model. Namely they considered 
\begin{equation*}
    G(t):=UA(t), \ t\in [-1,1]
\end{equation*}
with $A(t)=I_N-(1-t)vv^*$ and $U$ a random unitary matrix satisfying some assumptions (which are in particular satisfied by Haar matrices and the model of Theorem \ref{extension bis}). One can then consider the eigenvalues trajectories $\lambda_j(t)$ which, almost surely, do not intersect. They showed that for any $\alpha > 0$, with high probability, an outlier can be detected near $0$ when $\lvert t \lvert < N^{-\frac{1}{2}-\alpha}$. In contrast, for $\lvert t \lvert > N^{-\frac{1}{2}+\alpha}$, all eigenvalues lie, with high probability, very close to the unit circle. \\
They also showed that, for $U$ distributed according to the Haar measure, $N^{-1/2}$ is indeed the optimal timescale. Specifically, at time $t\sim N^{-1/2},$ there is no strongly separated outlier. Their proof relies crucially on the integrability of the model, which prevents this technique from being generalized directly. By observing that, at the critical timescale, the eigenvalues of the matrix correspond to the zeros of a sequence of random analytic functions converging to a Gaussian function, we can generalize their result. This approach establishes the absence of a strongly separated outlier at the critical timescale and extends to non-integrable models, thus offering a broader perspective while avoiding the limitations of integrability-based methods.

\smallskip 

\subsection*{Plan of the paper.}We begin in Section \ref{Section 1} by reviewing some results on the limit of random analytic functions which will be key for our analysis. A proof of Theorem \ref{Extension result}
is then developed in Section \ref{Section 2} that we generalize to another unitarily invariant model in Section \ref{Section 3}. The paper concludes in Section \ref{Section 4} with a demonstration that the timescale, which was proven to be optimal in the Haar setting, is also optimal for the model $M=VDV^{*}$ of Theorem \ref{extension bis}. 

\section*{Acknowledgments}
The author acknowledges support from the Foundation Mathématique Jacques Hadamard through the grant ANR-22-EXES-0013.
\section{Limit theorems for random analytic functions}\label{Section 1}
This section gathers general results that will be used later in the paper. Although the materiel is not new, we include it here with some minor adaptations and additions for the reader's convenience. Several proofs and formulations are inspired by \cite{Shirai}, to which we refer for further details and context.\\
\\
The core reasoning behind the proofs of Theorems \ref{Forrester-Ipsen result}, \ref{Extension result} and \ref{extension bis} is fairly straightforward. The eigenvalues of $UA$ can be characterized as the zeros of an analytic function for which we find a limit in law. The fact that the convergence in law of a sequence of random analytic functions implies the convergence of the associated empirical measures is the object of the next result. 
\begin{proposition}\label{zeros convergence}
    Let $\{f_N\}$ be a sequence of random analytic functions converging in law to a random analytic function $f$. Assume that, almost surely, $f$ is not identically zero. Then the sequence of associated empirical measures converges vaguely to $\mu_f$.
\end{proposition}
\begin{proof}
    It is sufficient to show 
    \begin{align*}
        F: &H(\D) \backslash \{0\}  \mapsto \mathcal{Q}\\
        & \ \ \ \ \ \ \ \ \ \ \ \  f\mapsto \sum_{w \in \mathcal{Z}_f} \de_w
    \end{align*}
    is continuous where $\mathcal{Q}$ is endowed with the vague topology and $H(\D)$ is endowed with the topology of uniform convergence on compact subset. Indeed, if $f_n \xrightarrow[N \to \infty]{d} f$ and $F$ is continuous, then $F(f_N) \xrightarrow[N \to \infty]{v} F(f).$ \\
    We now prove that $F$ is continuous. Let $f_N \rightarrow f$ in $H(\D) \backslash \{0\}$. Let $\phi \in \mathscr{C}_c(\D). $ We must show that 
    \begin{equation*}
        \sum_{w \in \mathcal{Z}_{f_n}}\phi(w)\xrightarrow[n\rightarrow\infty]{}\sum_{w\in \mathcal{Z}_f} \phi(w).
    \end{equation*}
    Let $K:=\mathrm{supp}(\phi)$. Denote $z_1,\ldots, z_p$ the distinct zeros of $f$ lying in $K$. Let $\de > 0$ small enough such that the disks $D(z_i,\de)$ are all disjoint. For all $n$ large enough we have that $f_n$ and $f$ have the same number of zeros in $D(z_i,\de), \ i=1,\ldots,p$ counting multiplicity. Let $L:=K\backslash \bigcup D(z_i,\de)$. Then we also have $\inf_{z\in L}\lvert f(z) \lvert >0$. Hence, for all $n$ large enough, by Rouché's Theorem, $f_n$ has no zero in $L$. Denote by $m$ the maximum of the multiplicities of $z_1,\ldots,z_p$ as zeros of $f$. Therefore for all $n$ large enough
    \begin{align*}
        \left\lvert \sum_{w \in \mathcal{Z}_{f_n}} \phi(w) \ \ - \sum_{w \in \mathcal{Z}_{f}} \phi(w) \right\rvert&= 
        \leq pm \max_{1\leq i \leq p} \sup_{z,w \in D(z_i, \de)}\lvert \phi(w)-\phi(z)\lvert \xrightarrow[\de \rightarrow 0]{}0.
    \end{align*}
    This concludes the proof.
\end{proof}

The following result will also be required (for a proof see \cite{Dja}, Lemma 3.2).

\begin{proposition}\label{limit thm}
Let $\{f_N\}$ be a tight sequence of random elements in $H(\D)$, and let us write, for every $N\geq 1$, $f_N(z):=\sum_{k\geq 0}P_k^{(N)}z^k$. If for every $m\geq0$, 
\begin{equation*}
    (P_0^{(N)},\ldots,P_m^{(N)})\xrightarrow[N\to \infty]{d}(P_0,\ldots,P_m)
\end{equation*}
for a common sequence of random variables $\{P_m\}$ then $f=\sum_{k\geq 0}P_kz^k$ is well-defined in $H(\D)$ and 
\begin{equation*}
    f_N \xrightarrow[N \to \infty]{d}f.
\end{equation*}
\end{proposition}
Note that the tightness is indeed necessary as illustrated by the deterministic example
\begin{equation*}
    h_N(z):=2^N z^N,\ z\in \D.
\end{equation*}
In this example all the coefficients converge to $0$ but $\{h_N\}$ as a sequence of analytic functions does not converge to $0$.
\smallskip

We now need some practical criterion for establishing tightness. If $f$ is an analytic function and $K$ is a compact set in $\D$ we note $\lvert \lvert f \lvert \lvert _{K}:=\sup_{z\in K}\lvert f(z)\lvert.$\\
Let $\{f_N\}$ be a sequence of random analytic functions. A sufficient condition for $\{f_N\}$ to be tight is that the sequence of real-valued random variables $\{\lvert \lvert f_N \lvert \lvert _K\}$ is tight for every compact set $K$ in $\D$ (see \cite{Shirai}, Proposition 2.5).\\
We now state the tightness criterion that we will use in the following sections.
For a compact set $K$ in $\D$ and $\de > 0$ we note $K_{\de}:=\{z \in \C,\  d(z,K)<\de\}$.
\begin{proposition}\label{prop tight}
    Let $\{f_N \}$ be a sequence of random analytic functions. Suppose that there exists $p\in \N$ such that the function $z\mapsto \sup_N \E[\lv f_N(z) \lv ^p]$ is locally integrable on $\D$. Then $\{f_N\}$ is tight in $H(\D)$.
\end{proposition}
\begin{proof}
    This follows directly from the next claim, proven below.  
    \begin{claim}\label{claim integral}
        Let $f \in H(\D)$ and let $K$ be a compact set in $\D$. Then for any $p \in \N$ and any sufficiently small $\de > 0$
        \begin{equation*}
            \lv \lv f \lv \lv_K ^p \leq (\pi \de^2)^{-1}\int_{\overline{K}_{\de}}\lv f(z) \lv ^p dm(z)
        \end{equation*}
        where $m$ is the Lebesgue measure on $\C$.
    \end{claim}
    Assuming the claim for now we have 
    \begin{equation*}
        \lv \lv f_N \lv \lv_K ^p \leq (\pi \de^2)^{-1}\int_{\overline{K}_{\de}}\lv f_N(z) \lv ^p dm(z).
    \end{equation*}
    Then by Fubini we get 
    \begin{align*}
        \E[\lv \lv f_N \lv \lv_K^p]&\leq (\pi \de^2)^{-1}\int_{\overline{K}_{\de}}\E[\lv f_N(z) \lv^p]dm(z)\\
        &\leq (\pi \de^2)^{-1}\int_{\overline{K}_{\de}}\sup_N\E[\lv f_N(z) \lv^p]dm(z)\\
        &< \infty 
    \end{align*}
    and so the $p$-th moment of the norms are uniformly bounded. Therefore by Markov's inequality the sequence $\{\lv \lv f_N \lv \lv _K\}$ is tight for any compact set $K$ in $\D$. As stated just before Proposition \ref{prop tight}, it implies the tightness of $\{f_N\}$ in $H(\D)$.
\end{proof}
\smallskip

We now prove the claim.
\begin{proofc}
    Let $f \in H(\D)$ and $K$ a compact set in $\D$. Let $p \in \N$ and $\de>0$ such that $\overline{K}_{\de} \subset \D$. We start by showing that, for $0 \leq r_1 < r_2 < \de$, for $a \in K$, 
    \begin{equation}\label{increasing integral}
        \frac{1}{2 \pi}\int_0^{2\pi}\lv f(a+r_1\re^{\ri \theta})\lv ^p d\theta \leq \frac{1}{2 \pi}\int_0^{2\pi}\lv f(a+r_2\re^{\ri \theta})\lv ^p d\theta.
    \end{equation}
    \smallskip
    
    For this we first show that $z \mapsto \lv f(z) \lv ^p$ is subharmonic in $\D$, that is that for any $a \in \D$, there exists $R_a > 0$ such that for any $0\leq r \leq R_a$
    \begin{equation*}
        \lv f(a) \lv^p \leq \frac{1}{2 \pi}\int_0^{2\pi}\lv f(a + r\re^{\ri \theta})\lv^p d\theta.
    \end{equation*}
    Since $f$ is analytic and $p$ is an integer, $f^p$ is also analytic and therefore satisfies the Mean Value Property
        \begin{equation*}
            f(a)^p=\frac{1}{2\pi}\int_0^{2\pi}f(a+r\re^{\ri \theta})^p d \theta
        \end{equation*}
        for $r$ small enough. The desired inequality easily follows.
    \smallskip
     
    Now let $0\leq r_1 \leq r_2 < \de$ and $a \in K$. Define $g(z):= \lv f(a+z) \lv  ^p$ for $z\in D(0,\de).$ Based on the preceding discussion we see that $g$ is subharmonic. Let $u$ be a harmonic function in $D(0,r_2)$ such that $u=g$ on $\partial D(0,r_2)$ (this is the classical Dirichlet problem, see for instance \cite{Jost}, Theorem 2.1.2). Now since $g$ is a continuous subharmonic function, $u$ is a harmonic function and $g=u$ on $\partial D(0,r_2)$ we have that $g\leq u$ in $D(0,r_2)$ (see for instance \cite{Edmunds}, Corollary 2.1.10). We can now perform the following computations
    \begin{align*}
        \frac{1}{2 \pi}\int_0^{2\pi}g(r_1 \re^{\ri \theta})d\theta & \leq \frac{1}{2\pi}\int_0^{2\pi}u(r_1 \re ^{\ri \theta})d\theta\\
        &=u(0)\\
        &=\frac{1}{2 \pi}\int_0^{2\pi}u(r_2 \re ^{\ri \theta})d\theta \\
        &=\frac{1}{2 \pi}\int_0^{2\pi}g(r_2 \re ^{\ri \theta})d\theta
    \end{align*}
    where at the second and fourth line we used the fact that an harmonic function satisfies the Mean Value Property. Recalling the definition of $g$ we see that this is exactly \eqref{increasing integral}.\\
    Taking $r_1=0$ and $r_2=r$  and multiplying by $r$ we get
    \begin{equation*}
        r\lvert f(a) \lvert ^p \leq \frac{1}{2 \pi}\int_0^{2 \pi} \lvert f(a+r\re^{\ri \theta})\lvert^p r d\theta.
    \end{equation*}
    Now integrating over $r$ it yields
    \begin{align*}
        \frac{\de^2}{2}\lv f(a) \lv^p & \leq \frac{1}{2\pi}\int_0^{\de}\int_0^{2\pi}\lv f(a+r\re^{\ri \theta})\lv^p d\theta r dr \\
        &\leq \frac{1}{2\pi}\int_{\lv a-z \lv < \de}\lv f(z) \lv ^p dm(z)\\
        &\leq \frac{1}{2 \pi}\int _{\overline{K}_{\de}}\lv f(z)\lv^p dm(z).
    \end{align*}
    Rearranging the inequality and taking the sup over $a$ in $K$ concludes the proof.
    \qed
\end{proofc}
\section{Multiplicative Perturbation Of  A Haar-Distributed Matrix}\label{Section 2}
Throughout this section $U$ will be an $N\times N$ Haar-distributed matrix and we define the matrix $A:=I_N-(1-aN^{-1/2})vv*$ with $a \in \C^{*}$ and $v$ a unit vector that can be random as long as it is independent from $U$ and distributed with a non-vanishing density on $S_{N-1}(\C)$. Using Weyl's inequalities it can be easily seen that the spectrum of $UA$ lies within the closed unit disk. In fact we can do slightly better.

\begin{lemma}\label{spectrum in the disk}
    The spectrum of $UA$ lies almost surely in the open unit disk. 
\end{lemma}
\begin{proof}
    First note that $(UA)^*UA=A^*A=I_N-(1-\lvert a \lvert^2N^{-1})vv^*$. Now let $\lambda$ an eigenvalue of $UA$ and $w$ a unit eigenvector associated to this eigenvalue. We then have 
    \begin{equation*}
        \lvert \lambda \lvert ^2= \langle UAw , UAw \rangle= \lvert a \lvert ^2 N^{-1}\lvert \widetilde{w_1}\lvert^2+ \lvert \widetilde{w_2}\lvert^2+\ldots + \lvert \widetilde{w_N}\lvert ^2
    \end{equation*}
    where the $\widetilde{w_i}$ are the coordinates of $w$ in an orthogonal basis that diagonalizes $(UA)^*(UA)$. We directly see that $\lvert \lambda \lvert \leq \lvert w\lvert=1$. Moreover $\lvert \lambda \lvert =1$ implies that $\widetilde{w_1}=0$ which then implies that $w$ is in the kernel of $vv^*$. Therefore 
    \begin{equation*}
        \lambda w=UAw=U(I_N-(1-aN^{-1/2})vv^*)w=Uw.
    \end{equation*}
    But then $w$ is an eigenvector of $U$ that lies in a subspace of positive codimension : this is an event of probability zero.
\end{proof}
\begin{rmk}\label{a.s spectrum}
    Notice that the previous argument only relies on the fact that almost surely no eigenvector of $U$ lies in a subspace of positive codimension. In particular the result still holds for the model $M=VDV^*$ which we study in Section \ref{Section 3}.
\end{rmk}
We now characterize the eigenvalues of $UA$ as the zeros of an analytic function whose coefficients depend on $U$ in a simple way.\\
We recall Sylvester's identity : for any $n\times m$ matrix $C$ and $m\times n$ matrix $B$ it holds that
\begin{equation*}
    \det(I_n + CB)=\det(I_m +BC).
\end{equation*}
.
\begin{lemma}\label{zeros eigenvalues}
    The eigenvalues of $UA$ are exactly the zeros of the following random analytic function 
    \begin{equation*}
        f_N(z):=a-\left(\scalebox{0.9}{$\sqrt{N}$}-a\right)\sum_{k\geq 1}\left(v^{*}(U^{*})^kv\right)z^k, \quad z\in \D.
    \end{equation*}
\end{lemma}
\begin{proof}
    By definition, $z$ is in the spectrum of $UA$ if and only if 
    \begin{equation*}
        0=\det\left(UA-z\right)=\det \left(U\left(I_N-\left(1-aN^{-1/2}\right)v v^{*}\right)-z\right)=\det\left(U-z-\left(1-aN^{-1/2}\right)Uvv^{*}\right).
    \end{equation*}
    By Lemma 7 we can restrict ourselves to $\lv z \lv <1$. Then the previous equation is equivalent to 
    \begin{equation*}
        \det\left(I_N-\left(1-aN^{-1/2}\right)(U-z)^{-1}Uvv^{*}\right)=0.
    \end{equation*}
    Using Sylvester's identity we find that $z$ is in the spectrum of $UA$ if and only if 
    \begin{equation*}
        1 - \left(1 - aN^{-1/2}\right)v^{*}\left(I_N-zU^{*}\right)^{-1}v=0.
    \end{equation*}
    Since $\lv z \lv <1$, we can perform the following series expansion
    \begin{equation*}
        1-\left(1-aN^{-1/2}\right)\sum_{k\geq 0}v^{*}(U^{*})^kvz^k=0.
    \end{equation*}
    By multiplying both sides by $\sqrt{N}$ we find that $z$ is in the spectrum of $UA$ if and only if
    \begin{equation*}
        f_N(z)=0
    \end{equation*}
    which concludes the proof.
\end{proof}
Our goal now is to apply Proposition \ref{limit thm}. We first address the finite-dimensional convergence of the coefficients, relying on the following result by Krishnapur in \cite{Krishnapur}, Lemma 10. Recall that $\{c_k\}$ is a sequence of independent standard complex Gaussians random variables.
\begin{lemma}\label{cv coefficients}
    The following convergence in law holds true 
    \begin{equation*}
        \scalebox{1.0}{$\sqrt{N}$}(e_1^{*}U^{*}e_1,e_1^{*}(U^{*})^2e_1,\ldots)\xrightarrow[N \to \infty]{d}(c_1,c_2,\ldots).
    \end{equation*}
    More precisely, any finite number of the random variables $\scalebox{0.9}{$\sqrt{N}$} e_1^{*}(U^{*})^ke_1$ converge in law to the corresponding Gaussian vector.
\end{lemma}
From this result we immediately deduce the desired convergence. Recall that $\phi_a$ is defined to be 
\begin{equation}\label{function limit}
    \phi_a(z)=a-\sum_{k\geq1}c_k z^k, \ z\in \D.
\end{equation}

It is well-defined in $H(\D)$ since, almost surely, $\limsup \lv c_k \lv^{1/k}=1$.

\begin{lemma}\label{convergence of the coefficient}
    The coefficients of $\{f_N\}$ converge in the sense of finite-dimensional law to the coefficients of $\phi_a$.
\end{lemma}
\begin{proof}
    The coefficients of $f_N$ are 
    \begin{equation*}
        P_0^{(N)}=a \ \text{and} \ P_k^{(N)}=-\scalebox{0.9}{$\sqrt{N}$}v^{*}(U^{*})^kv + av^{*}(U^{*})^kv.
    \end{equation*}
    Using the independence of $U$ and $v$, together with the unitary invariance of $U$, we obtain that for any integer $m$ the vector $\left(v^*U^*v,v^*(U^*)^2v,\ldots,v^*(U^*)^mv \right)$ has the same distribution as the vector $\left(e_1^*U^*e_1,e_1^*(U^*)^2e_1,\ldots,e_1^*(U^*)^me_1 \right)$. 
    Then by Lemma \ref{cv coefficients} we have that $\scalebox{1.0}{$\sqrt{N}$}v^{*}(U^{*})^kv \xrightarrow[N\to \infty]{d}c_k$ and so $v^{*}(U^{*})^kv \xrightarrow[N \to \infty]{d}0$.
    For any $m\geq 0$ we can write 
    \begin{equation*}
        (P_0^{(N)},\ldots,P_m^{(N)})= -\scalebox{1.0}{$\sqrt{N}$}(0,v^{*}U^{*}v,\ldots,v^{*}(U^{*})^mv)+(a,a v^{*}U^{*}v,\ldots,a v^{*}(U^{*})^mv).
    \end{equation*}
    Therefore we can conclude with Slutsky's Theorem.
\end{proof}
\smallskip
It remains to show the tightness of $\{f_N\}$. Our approach here diverges from that of \cite{ForIps}. Drawing inspiration from the computations in \cite{DubRek} we provide a direct proof of the tightness of $\{f_N\}$ using the criterion established in Proposition \ref{prop tight}. We denote the eigenvalues of $U$ by $\re^{\ri \theta_1},\ldots,\re^{\ri \theta_N}$ and the corresponding eigenvectors by $r_1,\ldots, r_N$ and we write $U=R\Delta R^{*}$ with $\Delta=\diag(\re^{\ri \theta_1},\ldots,\re^{\ri \theta_N})$.
\begin{lemma}\label{Tightness function}
    The sequence of random analytic functions $\{f_N\}$ is tight in $H(\D)$.
\end{lemma}
\begin{proof}
    Define 
    \begin{equation*}
        \widetilde{f_N}(z):=\scalebox{1.0}{$\sqrt{N}$}\sum_{k\geq 1}v^{*}(U^{*})^kv z^k, \quad z\in \D.
    \end{equation*}
    Note that $f_N(z)=a-\widetilde{f_N}(z)+aN^{-1/2}\widetilde{f_N}(z)$. It is then enough to show the tightness of $\{\widetilde{f_N}(z)\}$. Recalling the series expansion of $(I_N -zU^{*})^{-1}$ we see that
    \begin{equation*}
        \widetilde{f_N}(z)=\scalebox{1.0}{$\sqrt{N}$}v^{*}(zU^{*})(I_N-zU^{*})^{-1}v=\sqrt{N}v^{*}zR \Delta ^{*}(I_N-z\Delta ^{*})^{-1}R^{*}v .
    \end{equation*}
    Therefore
    \begin{equation*}
        \widetilde{f_N}(z)=\sum_{k=1}^N \scalebox{1.0}{$\sqrt{N}$}\lv \langle v,r_k \rangle \lv^2\frac{z\re^{-\ri \theta_k}}{1-z\re^{-\ri \theta_k}}=\sum_{k=1}^N \scalebox{1.0}{$\sqrt{N}$}\lv \langle v,r_k \rangle\lv^2\sum_{l\geq 1}(z\re^{-\ri \theta_k})^l.
    \end{equation*}
    Therefore
    \begin{align*}
        \E[\lv \widetilde{f_N}(z)\lv^2]&=\E \sum_{k_1,k_2=1}^N N\lv \langle v,r_{k_1} \rangle\lv^2 \lv \langle v,r_{k_2} \rangle\lv^2\sum_{l_1,l_2\geq1}z^{l_1}\overline{z}^{l_2}\re^{-\ri l_1 \theta_{k_1}}\re^{\ri l_2 \theta_{k_2}}\\
        &= \sum_{k_1,k_2=1}^N N\E[\lv \langle v,r_{k_1} \rangle\lv^2 \lv \langle v,r_{k_2} \rangle\lv^2]\sum_{l\geq1}\lv z \lv^{2l}\E[\re^{-\ri l (\theta_{k_1}-\theta_{k_2})}]
    \end{align*}
    where the independence of the eigenvectors and eigenvalues justifies splitting the expectation while unitary invariance causes the term with $l_1 \neq l_2$ to vanish.\ The independence of $R$ and $v$, combined with the unitary invariance of $R$ implies that $\left(\langle v, r_{k_1}\rangle,\langle v,r_{k_2}\rangle \right)$ has the same distribution as $\left(\langle e_1, r_{k_1}\rangle,\langle e_1,r_{k_2}\rangle \right)$. By Weingarten calculus we have 
    \begin{equation*}
        \E[\lv \langle e_1,r_{k_1}\rangle\lv^2 \lv \langle e_1,r_{k_2}\rangle\lv^2]=\frac{1+\de_{k_1 k_2}}{N(N+1)}
    \end{equation*}
    (see for instance Proposition 4.2.3 in \cite{Petz}).\\
    For the eigenvalues we invoke the fact (Proposition 3.11 in \cite{Meckes}) that 
    \begin{equation*}
        \E[\lv \tr(U^l)\lv^2]=\min(l,N)
    \end{equation*}
    from which it follows in particular that if $k_1 \neq k_2$
    \begin{equation*}
        \E[\re^{-\ri l (\theta_{k_1}-\theta_{k_2})}]=\frac{\min(l,N)-N}{N(N-1)}=-\frac{(N-l)_{+}}{N(N-1)}.
    \end{equation*}
    Finally, by splitting the sum according to $k_1=k_2$ and $k_1 \neq k_2$, we deduce that
    \begin{align*}
        \E[\lv \widetilde{f_N}(z)\lv^2]&=\sum_{k=1}^N \frac{2N}{N(N+1)}\sum_{l\geq 1}\lv z \lv ^{2l}-\sum_{k_1\neq k_2}\frac{N}{N(N+1)}\sum_{l\geq 1}\lv z \lv^{2l}\frac{(N-l)_{+}}{N(N-1)}\\
        &\leq \frac{2N^2}{N(N+1)}\frac{\lv z \lv ^2}{1-\lv z \lv^2}
        \leq \frac{2 \lv z \lv^2}{1-\lv z \lv^2}.
    \end{align*}
     Consequently we find that 
    \begin{equation*}
        \sup_{N}\E[\lv \widetilde{f_N}(z)\lv^2]\leq \frac{2 \lv z \lv^2}{1-\lv z \lv^2}.
    \end{equation*}
    The function $z\mapsto \frac{2 \lv z \lv^2}{1-\lv z \lv^2}$ being locally integrable in $\D$, Proposition \ref{prop tight} allows us to conclude that the sequence $\{\widetilde{f_N}(z)\}$ is tight.
\end{proof}
We now gather our preceding arguments and prove Theorem \ref{Extension result}.
\begin{theorem}
    The eigenvalues of $UA$ converge to the zeros of $\phi_a$. More precisely, 
    \begin{equation*}
        \sum_{\la \in \Sp(UA)}\de_{\la}\xrightarrow[N \to \infty]{v}\sum_{w\in \mathcal{Z}_{\phi_a}}\de_w
    \end{equation*}
    where the zeros are repeated with multiplicity.
\end{theorem}
\begin{proof}
    By Lemma \ref{zeros eigenvalues}, we know that the eigenvalues of $UA$ are precisely the zeros of $f_N$. Lemma \ref{convergence of the coefficient} establishes the finite-dimensional convergence of $\{f_N\}$ towards $\phi_{a}$, while Lemma \ref{Tightness function} ensures the tightness of $\{f_N\}$. Consequently, Proposition \ref{limit thm} implies that $\{f_N\}$ converges in law towards $\phi_{a}$. Finally, using Proposition \ref{zeros convergence} we conclude the desired convergence of the eigenvalues. 
\end{proof}

\bigskip
\section{Study of a unitarily invariant model with i.i.d eigenvalues}\label{Section 3}
We now consider a model previously studied in \cite{DubRek}. Throughout this section $M$ is defined as 
\begin{equation}\label{deuxieme modele}
    M=VDV^{*}
\end{equation}
where $V$ and $D$ are independent square matrix of size $N$, $V$ is a Haar-distributed matrix and $D=\diag(Z_1,\ldots,Z_N)$ with $\{Z_k\}$ i.i.d on the unit circle such that $\sup_{k \geq 1} \lvert \E [Z_1^k]\lvert = o(N^{-b})$ for any integer $b$ and such that the spectrum of $M$ is simple almost surely. Recall the definition of  $A=I_N-(1-aN^{-1/2})vv^*$ where $a\in \C^*$ and $v$ is a unit vector that can be random as long as it is independent of $M$ and distributed with a non-vanishing density on $S_{N-1}(\C)$.\\
Following the same strategy as in Section \ref{Section 2} we prove that the eigenvalues of $MA$ converge to the zeros of the following random analytic function 
\begin{equation}\label{function g}
    \phi_{a'}(z)= a' - \sum_{k\geq1}c_kz^k, \ \ a'=\frac{a}{\sqrt{2}}
\end{equation}
where the $\{c_k\}$ are complex independent standard Gaussian random variables.
\begin{lemma}
    The spectrum of $MA$ lies almost surely in the open unit disk.
\end{lemma}
\begin{proof}
    Note that the eigenvectors of $M$ are the columns of $V$ and so the probability of one of them lying in a subspace of positive codimension is zero. The rest of the proof is exactly the same as in Lemma \ref{spectrum in the disk}.
\end{proof}
\begin{lemma}\label{zeros eigenvalues v2}
    The eigenvalues of $MA$ are exactly the zeros of the following random analytic function 
    \begin{equation*}
        g_N(z):=a'-\left(\sqrt{\frac{N}{2}} - a'\right)\sum_{k\geq 1}v^{*}(U^{*})^k v z^k, \ z\in \D.
    \end{equation*}
\end{lemma}
\begin{proof}
    Doing exactly the same as in the proof of Lemma \ref{zeros eigenvalues} we arrive at 
    \begin{equation*}
        z \in \Sp(MA) \iff 1 - \left(1-aN^{-1/2}\right)\sum_{k\geq 0}v^{*}(U^{*})^kvz^k=0.
    \end{equation*}
    By multiplying both sides by $\sqrt{\frac{N}{2}}$ we find that $z$ is in the spectrum of $MA$ if and only if 
    \begin{equation*}
        g_N(z)=0
    \end{equation*}
    which concludes the proof.
\end{proof}
The following lemma resolves the issue of coefficients convergence.
\begin{lemma}\label{convergence coeff v2}
    The following convergence in law holds true 
    \begin{equation*}
        \sqrt{\frac{N}{2}}(e_1^{*}U^{*}e_1,e_1^{*}(U^{*})^2e_1,\ldots)\xrightarrow[N \to \infty]{d}(c_1,c_2,\ldots).
    \end{equation*}
    In particular the coefficients of $\{g_N\}$ converge to the coefficients of $\phi_{a'}$.
\end{lemma}
The convergence will be proven with the moments' method.\ The fact that it implies the convergence of the coefficients of $\{g_N\}$ is exactly the same argument as in Lemma \ref{convergence of the coefficient} and we do not repeat it. Due to the technical nature of the proof, and for the sake of readability, the detailed argument is deferred at the end of the section.\\
It remains to show the tightness of $\{g_N\}$. The following lemma is proved in a manner analogous to Lemma \ref{Tightness function}. We denote by $\widetilde{v_1},\ldots,\widetilde{v_N}$ the columns of $V$.
\begin{lemma}\label{tightness v2}
    The sequence of random analytic functions $\{g_N\}$ is tight in $H(\D)$.
\end{lemma}
\begin{proof}
    Define 
    \begin{equation*}
        \widetilde{g_N}(z):=\scalebox{1.0}{$\sqrt{N}$}\sum_{k\geq 1}v^{*}(U^{*})^kv z^k, \quad z\in \D.
    \end{equation*}
    Note that $g_N(z)=a'-\frac{\widetilde{g_N}(z)}{\sqrt{2}} + a'N^{-1/2}\widetilde{g_N}(z).$ It is then enough to show the tightness of $\{\widetilde{g_N}\}$. Recalling the series expansion of $(I_N-zU^{*})^{-1}$ we see that 
    \begin{equation*}
        \widetilde{g_N}(z)=\scalebox{1.0}{$\sqrt{N}$}v^{*}(zU^{*})(I_N-zU^{*})^{-1}v=\sqrt{N}v^{*}zV D ^{*}(I_N-zD^{*})^{-1}V^{*}v .
    \end{equation*}
    Therefore 
    \begin{equation*}
        \widetilde{g_N}(z)=\sqrt{N}\sum_{k=1}^N \lvert \langle v,\widetilde{v_k} \rangle\lvert ^2 \frac{z \overline{Z_k}}{1-z\overline{Z_k}}=\sqrt{N}\sum_{k=1}^N \lvert \langle v,\widetilde{v_k}\rangle\lvert ^2 \sum_{l\geq1}(z\overline{Z_k})^l.
    \end{equation*}
    This yields that 
    \begin{align*}
        \E[\lvert \widetilde{g_N}(z)\lvert^2]&= \E \sum_{k_1,k_2=1}^N N \lvert \langle v,\widetilde{v_{k_1}} \rangle\lvert ^2 \lvert \langle v,\widetilde{v_{k_2}} \rangle\lvert^2 \sum_{l_1,l_2 \geq 1} z^{l_1}\overline{z}^{l_2} \overline{Z_{k_1}}^{l_1}Z_{k_2}^{l_2}\\
        &=\sum_{k_1,k_2=1}^N N \E [\lvert \langle v,\widetilde{v_{k_1}} \rangle\lvert ^2 \lvert \langle v,\widetilde{v_{k_2}} \rangle\lvert^2 ]\sum_{l_1,l_2 \geq 1} z^{l_1}\overline{z}^{l_2} \E[\overline{Z_{k_1}}^{l_1}Z_{k_2}^{l_2}]\\ 
    \end{align*}
    The independence of $V$ and $v$, combined with the unitary invariance of $V$ implies that $\left(\langle v, \widetilde{v_{k_1}}\rangle,\langle v,\widetilde{v_{k_2}}\rangle \right)$ has the same distribution as $\left(\langle e_1, r_{k_1}\rangle,\langle e_1,r_{k_2}\rangle \right)$. By Weingarten calculus we have 
    \begin{equation*}
        \E[\lv \langle e_1,r_{k_1}\rangle\lv^2 \lv \langle e_1,r_{k_2}\rangle\lv^2]=\frac{1+\de_{k_1 k_2}}{N(N+1)}
    \end{equation*}
    (see for instance Proposition 4.2.3 in \cite{Petz}). 
    Therefore
    \begin{align*}
        \E[\lvert \widetilde{g_N}(z)\lvert^2]&\leq \sum_{k=1}^N \frac{2}{N+1}\left( \sum_{l\geq 1}\lvert z \lvert^{2l} \ \ + \sum_{l_1 \neq l_2} \lvert z \lvert ^{l_1+l_2}\sup_{m\geq 1}\lvert \E [Z_1^m]\lvert \right)\\
        &+\sum_{1\leq k_1 \neq k_2 \leq N} \frac{1}{N+1} \sum_{l_1,l_2 \geq 1} \lvert z \lvert ^{l_1+l_2} \left(\sup_{m\geq 1}\lvert \E [Z_1^m]\lvert \right)^2\\
        &\leq \frac{2\lvert z \lvert ^2}{1-\lvert z \lvert ^2} + C\left(\frac{\lvert z \lvert}{1- \lvert z \lvert}\right)^2 
    \end{align*}
    where $C$ is a constant. Consequently we find that 
    \begin{equation*}
        \sup_{N}\E[\lvert \widetilde{g_N}(z)\lvert ^2]\leq \frac{2\lvert z \lvert ^2}{1-\lvert z \lvert ^2} + C\left(\frac{\lvert z \lvert}{1- \lvert z \lvert}\right)^2 .
    \end{equation*}
    The function $z \mapsto \frac{2\lvert z \lvert ^2}{1-\lvert z \lvert ^2} + C\left(\frac{\lvert z \lvert}{1- \lvert z \lvert}\right)^2 $ being locally integrable in $\D$, Proposition \ref{prop tight} allows us to conclude that the sequence $\{\widetilde{g_N}\}$ is tight.
\end{proof}
We now gather our preceding arguments and prove Theorem \ref{extension bis}.
\begin{theorem}\label{iid model result}
    The eigenvalues of $MA$, where $M=VDV^{*}$, converge to the zeros of $\phi_{a'}$. More precisely 
    \begin{equation*}
        \sum_{\la \in \Sp(MA)}\de_{\la} \xrightarrow[N \to \infty]{v}\sum_{w \in \mathcal{Z}_{\phi_{a'}}} \de_w
    \end{equation*}
    where the zeros are repeated with multiplicity.
\end{theorem}
\begin{proof}
    By Lemma \ref{zeros eigenvalues v2}, we know that the eigenvalues of $MA$ are precisely the zeros of $g_N$. Lemma \ref{convergence coeff v2} establishes the finite-dimensional convergence of $\{g_N\}$ towards $\phi_{a'}$, while Lemma \ref{tightness v2} ensures the tightness of $\{g_N\}$. Consequently, Proposition \ref{limit thm} implies that $\{g_N\}$ converges in law towards $\phi_{a'}$. Finally, using Proposition \ref{zeros convergence} we conclude the desired convergence.
\end{proof}

We now prove Lemma \ref{convergence coeff v2}. We first introduce some notations. For $p \in \N$ we note $[p]:=\{1,\ldots,p\}.$ For $i=(i(1),\ldots,i(p))$ a $p$-uplet of numbers in $[N]$ we denote by $\#i$ the cardinal of the set $i([p])$. In the proof we shall make use of the following result on the joint moments of entries of a unitary matrix (see \cite{Speicher} page 381 and \cite{Krishnapur} Result 15).
\begin{lemma}\label{weingarten}
    Let $V=(v_{ij})_{i,j\leq N}$ be a Haar matrix. Let $k\leq N$ and fix $i(l),j(l),i'(l),j'(l)$ for $1\leq l \leq k$. Then 
    \begin{equation*}
        \E \left[  \prod_{l=1}^k v_{i(l)j(l)}\prod_{l=1}^k \overline{v}_{i'(l)j'(l)}  \right]=\sum_{\pi,\sigma \in \mathfrak{S}_k}\mathrm{Wg}(N,\pi \sigma^{-1})\prod_{l=1}^k \mathds{1}_{i(l)=i'(\pi(l))}\mathds{1}_{j(l)=j'(\sigma(l))}
    \end{equation*}
    where the Weingarten function $\mathrm{Wg}$ has the property that as $N\rightarrow \infty$
    \[\mathrm{Wg}(N, \tau) = 
\begin{cases} 
N^{-k} + O(N^{-k-1}) & \text{if } \tau = e \, (\text{``identity''}). \\
O(N^{-k-1}) & \text{if } \tau \neq e.
\end{cases}\]

\end{lemma}
In particular for any $i(l),j(l),i'(l),j'(l)$
\begin{align*}\label{estimation O}
    \left \lvert \E \left[  \prod_{l=1}^k v_{i(l)j(l)}\prod_{l=1}^k \overline{v}_{i'(l)j'(l)}  \right]\right \lvert&\leq \sum_{\pi,\sigma \in \mathfrak{S}_k}\left \lv \mathrm{Wg}(N,\pi \sigma^{-1})\right \lv \\
     &\leq \max_{\pi, \sigma \in \mathfrak{S}_k}\left \lv \mathrm{Wg}(N,\pi \sigma^{-1})\right \lv(k!)^2  \\
     &= O(N^{-k})
\end{align*}
independant of $i,j,i',j'$.

\begin{proof}[Proof Lemma 15]
    For $n$ an integer we denote 
    \begin{align*}
        S_n:&=\sqrt{\frac{N}{2}}e_1^{*}U^ne_1=\sqrt{\frac{N}{2}}\sum_{i=1}^N \lvert v_{1i}\lvert ^2 Z_i^n
    \end{align*}
    where the $\{Z_i\}$ are independent and identically distributed random variables on the unit circle. Therefore 
    \begin{equation*}
        \overline{S_n}=\sqrt{\frac{N}{2}}e_1^{*}(U^*)^ne_1.
    \end{equation*}
    Let $n_1<\ldots<n_l \ ; \ n'_1<\ldots<n'_r \ ; \ p_1,\ldots,p_l \ ; \ q_1, \ldots , q_r$ all integers. We must show that
    \begin{equation*}
        \underset{N \to \infty}{\lim}\E[S_{n_1}^{p_1}\ldots S_{n_l}^{p_l}\overline{S_{n'_1}^{q_1}}\ldots\overline{S_{n'_r}^{q_r}} ] =\mathds{1}_{l=r}\left( \prod_{k=1}^l \mathds{1}_{p_k=q_k}\mathds{1}_{n_k=n'_k}  \right) p_1!\ldots p_l!.
    \end{equation*}
    We denote by $p:=p_1+\ldots+p_l$ and by $q:=q_1+\ldots +q_r$.
    Then 
    \begin{align*}
        \E&[S_{n_1}^{p_1}\ldots S_{n_l}^{p_l}\overline{S_{n'_1}^{q_1}}\ldots\overline{S_{n'_r}^{q_r}} ]\\
        &=\left(\frac{N}{2}\right)^{\frac{p+q}{2}}\sum_{\substack{i_1 \ : \ [p_1]\rightarrow [N] \\ \vdots \\i_l \ : \ [p_l] \rightarrow [N] \\ j_1 \ : \ [q_1] \rightarrow [N] \\ \vdots \\ j_r \ : \ [q_r] \rightarrow [N]}} \E [\lvert v_{1i_1(1)}\lvert^2 \ldots \lvert v_{1i_l(p_l)}\lvert^2 \lvert v_{1j_1(1)}\lvert^2 \ldots \lvert v_{1j_r(q_r)}\lvert^2]\E [Z_{i_1(1)}^{n_1}\ldots Z_{i_l(p_l)}^{n_l}\overline{Z_{j_1(1)}^{n'_1}} \ldots \overline{Z_{j_r(q_r)}^{n'_r}}]
    \end{align*}
 
    For notational convenience, let us denote $\sup_{k\geq1}\lvert \E[Z_1^k]\lvert = b_N$ and recall that $b_N=O(N^{-b})$ for any integer $b$. Note that if $n_1p_1+\ldots +n_lp_l \neq n'_1q_1+\ldots n'_rq_r$ then
    \begin{equation*}
        \left \lvert \E[Z_{i_1(1)}^{n_1}\ldots Z_{i_1(p_1)}^{n_1}\ldots Z_{i_l(1)}^{n_l}\ldots Z_{i_l(p_l)}^{n_l}\overline{Z_{j_1(1)}^{n'_1}}\ldots \overline{Z_{j_1(q_1)}^{n'_1}}\ldots \overline{Z_{j_r(1)}^{n'_r}}\ldots \overline{Z_{j_r(q_r)}^{n'_r}}]\right \lvert \leq b_N.
    \end{equation*}
    Therefore, in this case, using Lemma \ref{weingarten}, we have 
    \begin{equation*}
        \E[S_{n_1}^{p_1}\ldots S_{n_l}^{p_l}\overline{S_{n'_1}^{q_1}}\ldots\overline{S_{n'_r}^{q_r}} ] = O(N^{\frac{p+q}{2}}b_N) \xrightarrow[N \to \infty]{}0.
    \end{equation*}
    Assume now that $n_1p_1+\ldots +n_lp_l = n'_1q_1+\ldots n'_rq_r :=L$. It will be convenient to rewrite 
    \begin{equation*}
        Z_{i_1(1)}^{n_1}\ldots Z_{i_1(p_1)}^{n_1}\ldots Z_{i_l(1)}^{n_l}\ldots Z_{i_l(p_l)}^{n_l}\overline{Z_{j_1(1)}^{n'_1}}\ldots \overline{Z_{j_1(q_1)}^{n'_1}}\ldots \overline{Z_{j_r(1)}^{n'_r}}\ldots \overline{Z_{j_r(q_r)}^{n'_r}} = Z_{i(1)}\ldots Z_{i(L)}\overline{Z_{j(1)}}\ldots \overline{Z_{j(L)}}
    \end{equation*}
    where 
    \begin{align*}
        i(1)&=\ldots=i(n_1)=i_1(1)\\
        i(n_1 +1)&=\ldots=i(2n_1)=i_1(2)\\
        & \ \   \  \ \  \  \vdots \\
         i(n_1p_1+\ldots+n_{l-1}p_{l-1} +1)&=\ldots=i(n_1p_1+\ldots+n_{l-1}p_{l-1} +n_l)=i_l(1)\\
        & \                \ \ \   \ \  \vdots \\
        i(n_1p_1+\ldots + n_l(p_l -1)+1)&=\ldots=i(L)=i_l(p_l)
    \end{align*}
    and 
    \begin{align*}
        j(1)&=\ldots=j(n'_1)=j_1(1)\\
        j(n'_1 +1)&=\ldots=i(2n'_1)=j_1(2)\\
        & \ \ \ \ \               \     \vdots \\
         j(n'_1q_1+\ldots+n'_{r-1}q_{r-1} +1)&=\ldots=j(n'_1q_1+\ldots+n'_{r-1}q_{r-1} +n'_r)=j_r(1)\\
        & \ \ \ \ \               \    \vdots \\
        j(n'_1q_1+\ldots + n'_r(q_r -1)+1)&=\ldots=j(L)=j_r(q_r)
    \end{align*}
    \\ 
    By construction $\#i \leq p$ and $\#j \leq q$. Knowledge of $i,j$ is equivalent to the knowledge of $i_1,\ldots,i_l,j_1,\ldots,j_r$ and we will freely alternate between the two notations.
    \\
    \\
    We will denote by $S$ the set of $L$-tuples $i $ of numbers in $[N]$ that could be a rewriting of some $i_1,\ldots, i_l$, that is, that verifies 
    \begin{align*}
        i(1)&=\ldots=i(n_1)\\
        i(n_1 +1)&=\ldots=i(2n_1)\\
        &\ \  \ \ \ \   \vdots \\
        i(n_1p_1+\ldots+n_{l-1}p_{l-1} +1)&=\ldots=i(n_1p_1+\ldots+n_{l-1}p_{l-1} +n_l)\\
        & \ \ \ \ \ \  \vdots \\
        i(n_1p_1+\ldots + n_l(p_l -1)+1)&=\ldots=i(L)
    \end{align*}
    and respectively $S'$ for $j$ and $j_1,\ldots,j_r.$ We can write
    \begin{equation}\label{Esp 2}
        \E[S_{n_1}^{p_1}\ldots S_{n_l}^{p_l}\overline{S_{n'_1}^{q_1}}\ldots\overline{S_{n'_r}^{q_r}} ]=\left(\frac{N}{2}\right)^{\frac{p+q}{2}}\sum_{(i,j)\in S\times S'}\Psi_{ij}\Theta_{ij}
    \end{equation}
    where 
    \begin{align}\label{psi}
        \Psi_{ij}&=\E[\lvert v_{1i_1(1)}\lvert ^2 \ldots \lvert v_{1i_l(p_l)}\lvert^2 \lvert v_{1j_1(1)}\lvert ^2 \ldots \lvert v_{1j_r(q_r)}\lvert^2] \\
        &=O(N^{-(p+q)})   
    \end{align}
    thanks to Lemma \ref{weingarten}, and 
    \begin{equation*}
        \Theta_{ij}=\E[Z_{i(1)}\ldots Z_{i(L)}\overline{Z_{j(1)}}\ldots \overline{Z_{j(L)}}].
    \end{equation*}
    \\
    Notice that, for $i: [L]\rightarrow [N]$ in $S$ fixed, $\Theta_{ij}=O(b_N)$ unless there exists $\sigma \in \mathfrak{S}_L$ such that $j\circ \sigma =i$ and in this case $\Theta_{ij}=1.$
    \\
    Therefore if $i\in S$ is fixed then
    \begin{align*}
        \sum_{j \in S'}\Psi_{ij}\Theta_{ij}&=O(N^{-(p+q)}) \#\big(\{j:[L]\rightarrow[N], \exists \sigma \in \mathfrak{S}_L, j\circ \sigma =i\}\cap S' \big)\\
        &+ O(N^{-(p+q)})\#\big(\{j:[L]\rightarrow[N], \exists \sigma \in \mathfrak{S}_L, j\circ \sigma =i\}^c\cap S' \big)O(b_N)
    \end{align*}
    A moment of reflexion yields that
    \begin{align*}
        \#\{j:[L]\rightarrow[N], \exists \sigma \in \mathfrak{S}_L, j\circ \sigma =i \}&=\frac{L!}{(\#i^{-1}(1))!\ldots(\#i^{-1}(N))!}\\
        &=O(1).
    \end{align*}
    Indeed to construct such a $j$ we have to choose $\# i^{-1}(1)$ elements in $[L]$ such that their image under $j$ will be $1$ and we have to choose $\# i^{-1}(2)$ elements in $[L] \backslash i^{-1}(\{1\}) $   so on for $3,4$ until $N$. Therefore
    \begin{align*}
         \#\{j:[L]\rightarrow[N], \exists \sigma \in \mathfrak{S}_L, j\circ \sigma =i \}&= \binom{L}{\#i^{-1}(1)}\ldots \binom{L-(\#i^{-1}(1)+\ldots + \#i^{-1}(N-1))}{\#i^{-1}(N)}\\
         &=\frac{L!}{(\#i^{-1}(1))!\ldots(\#i^{-1}(N))!}.
    \end{align*}
    Also $\# S'=N^{p+q}$ and recall that $b_N$ decays faster than any polynomial.
    This implies that 
    \begin{equation*}
        \sum_{j \in S'}\Psi_{ij}\Theta_{ij}=O(N^{-(p+q)})
    \end{equation*}
    with $O(N^{-(p+q)})$ independent of $i$. Returning to \eqref{Esp 2} we get 
    \begin{align*}
        \E[S_{n_1}^{p_1}\ldots S_{n_l}^{p_l}\overline{S_{n'_1}^{q_1}}\ldots\overline{S_{n'_r}^{q_r}} ]&=\left(\frac{N}{2}\right)^{\frac{p+q}{2}}\sum_{i\in S}O(N^{-(p+q)})\\
        &=\left(\frac{N}{2}\right)^{\frac{p+q}{2}}\sum_{m=1}^{p}\sum_{\substack{i\in S \\ \#i=m}}O(N^{-(p+q)})
    \end{align*}
    Without loss of generality we can assume that $q\geq p.$ If $m<p$ then
    \begin{align*}
        \left(\frac{N}{2}\right)^{\frac{p+q}{2}}\sum_{\substack{i\in S \\ \#i=m}}O(N^{-(p+q)})&=O(N^{\frac{p+q}{2}}N^{-(p+q)}N^m)
        =O(N^{m-\frac{p+q}{2}})
        =O(N^{-1})
    \end{align*}
    since $\frac{p+q}{2} \geq p.$ Therefore
    \begin{equation*}
        \E[S_{n_1}^{p_1}\ldots S_{n_l}^{p_l}\overline{S_{n'_1}^{q_1}}\ldots\overline{S_{n'_r}^{q_r}} ]= \left(\frac{N}{2}\right)^{\frac{p+q}{2}}\sum_{\substack{i\in S \\ \#i=p}}\sum_{j\in S'}\Psi_{ij}\Theta_{ij}  \ \  + \ \  O(N^{-1}).
    \end{equation*}
    Suppose that $q > p.$ Then noting that 
    \begin{equation*}
        \#\{i\in S, \ \#i=p\}=O(N^{p})
    \end{equation*}
    we have
    \begin{align*}
        \left(\frac{N}{2}\right)^{\frac{p+q}{2}}\sum_{\substack{i\in S \\ \#i=p}}\sum_{j\in S'}\Psi_{ij}\Theta_{ij}&= \left(\frac{N}{2}\right)^{\frac{p+q}{2}}\sum_{\substack{i\in S \\ \#i=p}} O(N^{-(p+q)}) \\
        &=O(N^{\frac{p+q}{2}}N^{-(p+q)}N^{p})\\
        &=O(N^{\frac{p-q}{2}})\\
        &=O(N^{-\frac{1}{2}}).
    \end{align*}
    To have a non-zero limit it is then necessary that $p=q$.
    In this case 
    \begin{equation*}
        \E[S_{n_1}^{p_1}\ldots S_{n_l}^{p_l}\overline{S_{n'_1}^{q_1}}\ldots\overline{S_{n'_r}^{q_r}} ]=\left(\frac{N}{2}\right)^{p}\sum_{\substack{i\in S \\ j\in S'\\\#i=p \\
        }}\Psi_{ij}\Theta_{ij} \ \ + \ \ O(N^{-1}).
    \end{equation*}
    Now as we already did for $i$ we will see that the only $j$ that have an asymptotic contribution are those such that $\#j=q=p$. Recall that $\Theta_{ij}=O(b_N)$ unless there exists $\sigma \in \mathfrak{S}_L$ such that $j\circ \sigma =i$. Now if $m<q=p, \ \#i=p$ and $\#j=m$ then there can be no $\sigma$ such that $j\circ \sigma =i$. Therefore
    \begin{align*}
        \left(\frac{N}{2}\right)^{p}\sum_{\substack{i\in S \\ j\in S'\\ \#j=m \\ \#i=p \\
        }}\Psi_{ij}\Theta_{ij}=O(N^sb_N)
    \end{align*}
    where $s$ is an integer that depends on $p$ and $m$.
    Hence
    \begin{equation*}
        \E[S_{n_1}^{p_1}\ldots S_{n_l}^{p_l}\overline{S_{n'_1}^{q_1}}\ldots\overline{S_{n'_r}^{q_r}} ]= \left(\frac{N}{2}\right)^{\frac{p+q}{2}}\sum_{\substack{i\in S \\ \#i=p}}\sum_{\substack{j\in S' \\ \#j=p}}\Psi_{ij}\Theta_{ij}  \ \  + \ \  O(N^{-1}).
    \end{equation*}
    Recall that $i$ was defined as the "concatenation" of $i_1,\ldots,i_l$ and note that 
    \begin{equation*}
        i\in S, \#i=p \Leftrightarrow \forall k\in [l], \#i_k=p_k, \ i_k([p_k])\cap i_s([p_s])=\emptyset \ \  \text{if} \ \ k\neq s.
    \end{equation*}
    Likewise 
    \begin{equation*}
        j\in S', \#j=q \Leftrightarrow \forall k\in [r],\ \#j_k=q_k,\ j_k([q_k])\cap j_s([q_s])=\emptyset \ \  \text{if} \ \ k\neq s.
    \end{equation*}
    Denote 
    \begin{equation*}
        T:=\{(i_1,\ldots,i_l), \  i_k : [p_k]\rightarrow[N], \ \#i_k=p_k,\ i_k([p_k])\cap i_s([p_s])=\emptyset \ \  \text{if} \ \ k\neq s\}
    \end{equation*}
    and 
    \begin{equation*}
        T':=\{(j_1,\ldots,j_r), \  j_k : [q_k]\rightarrow[N], \ \#j_k=q_k,\ j_k([q_k])\cap j_s([q_s])=\emptyset \ \  \text{if} \ \ k\neq s\}.
    \end{equation*}
    Then we have 
    \begin{equation*}
        \E[S_{n_1}^{p_1}\ldots S_{n_l}^{p_l}\overline{S_{n'_1}^{q_1}}\ldots\overline{S_{n'_r}^{q_r}} ]=\left(\frac{N}{2}\right)^{p} \sum_{(i_1,\ldots,i_l)\in T}\sum_{(j_1,\ldots,j_r)\in T'}\Psi_{ij}\Theta_{ij} \ \ + \ \ O(N^{-1}).
    \end{equation*}
    Fix $(i_1,\ldots,i_l)\in T$ and $(j_1,\ldots,j_r)\in T'$. Recall that 
    \begin{equation*}
        \Theta_{ij}=\E[Z_{i_1(1)}^{n_1}\ldots Z_{i_1(p_1)}^{n_1}\ldots Z_{i_l(1)}^{n_l}\ldots Z_{i_l(p_l)}^{n_l}\overline{Z_{j_1(1)}^{n'_1}}\ldots \overline{Z_{j_1(q_1)}^{n'_1}}\ldots \overline{Z_{j_r(1)}^{n'_r}}\ldots \overline{Z_{j_r(q_r)}^{n'_r}}].
    \end{equation*}
    Since the $i_s(t)$ are all different there can be no possible combination between the $Z_{i_s(t)}^{n_s}$ in order to compensate a $\overline{Z_{j_r(t)}^{n'_r}}$ and likewise there can be no possible combination between the  $\overline{Z_{j_r(t)}^{n'_r}}$ to compensate a $Z_{i_s(t)}^{n_s}$  . Therefore $\left \lvert \Theta_{ij}\right \lvert \leq b_N$ unless $l=r, \ p_1=q_1,\ldots,p_l=q_l,n_1=n'_1,\ldots,n_l=n'_l$ and there exist $\sigma_1\in \mathfrak{S}_{p_1},\ldots,\sigma_{l}\in \mathfrak{S}_{p_l}$ such that $j_1 \circ \sigma_1=i_1,\ldots,j_l \circ \sigma_l=i_l$ and if this is the case then $\Theta_{ij}=1$. Thus, in the case where one of the conditions $l=r, \ p_1=q_1,\ldots,p_l=q_l,n_1=n'_1,\ldots,n_l=n'_l$ is not satisfied we have 
    \begin{equation*}
        \E[S_{n_1}^{p_1}\ldots S_{n_l}^{p_l}\overline{S_{n'_1}^{q_1}}\ldots\overline{S_{n'_r}^{q_r}} ]= O(N^sb_N) +O(N^{-1})\xrightarrow[N \to \infty]{}0
    \end{equation*}
    where $s$ is some integer that depend on $p$ and $q$.\\
    Assume now that $l=r, \ n_1=n'_1,\ldots,n_l=n'_l, \ p_1=q_1,\ldots,p_l=q_l.$ 
    For $i=(i_1,\ldots,i_l)\in T$ fixed note 
    \begin{equation*}
        T_i ':=\{(j_1, \ldots,j_l)\in T', \forall k \in [l], \ \exists \sigma_k\in \mathfrak{S}_{p_k},\ j_k \circ \sigma_k=i_k\}.
    \end{equation*}
    If $i=(i_1,\ldots,i_l)\in T$ and $j=(j_1,\ldots,j_l)\in T_{i}'$ then
    \begin{align*}
        \Psi_{ij}&=\E[\lvert v_{1i_1(1)}\lvert ^2 \ldots \lvert v_{1i_l(p_l)}\lvert^2 \lvert v_{1j_1(1)}\lvert ^2 \ldots \lvert v_{1j_l(p_l)}\lvert^2]\\
        &=\E[\lvert v_{1i_1(1)}\lvert^4\ldots \lvert v_{1i_l(p_l)}\lvert^4].
    \end{align*}
    Therefore 
    \begin{align*}
        \E[S_{n_1}^{p_1}\ldots S_{n_l}^{p_l}\overline{S_{n'_1}^{q_1}}\ldots\overline{S_{n'_r}^{q_r}} ]&=\left(\frac{N}{2}\right)^{p} \sum_{(i_1,\ldots,i_l)\in T}\sum_{j\in T_{i}'}\Psi_{ij}\Theta_{ij} \ \ + \ \ O(N^{-1})\\
        &=\left(\frac{N}{2}\right)^{p} \sum_{(i_1,\ldots,i_l)\in T}\E[\lvert v_{1i_1(1)}\lvert^4\ldots \lvert v_{1i_l(p_l)}\lvert^4] \#T_{i}' \ \ + \ \ O(N^{-1})\\
        &=\left(\frac{N}{2}\right)^{p} \sum_{(i_1,\ldots,i_l)\in T}\E[\lvert v_{1i_1(1)}\lvert^4\ldots \lvert v_{1i_l(p_l)}\lvert^4] p_1!\ldots p_l! \ \ + \ \ O(N^{-1}).
    \end{align*}
    For $i=(i_1,\ldots,i_l)\in T$ we have by Lemma \ref{weingarten}
    \begin{align*}
        \E[\lvert v_{1i_1(1)}\lvert^4\ldots \lvert v_{1i_l(p_l)}\lvert^4]&=\E[v_{1i_1(1)}^2\ldots v_{1i_l(p_l)}^2 \overline{v}_{1i_1(1)}^2\ldots \overline{v}_{1i_l(p_l)}^2]\\
        &=\E[v_{1m(1)}\ldots v_{1m(2p)}\overline{v_{1m(1)}\ldots v_{1m(2p)}}]\\
        &=\sum_{\pi,\sigma \in \mathfrak{S}_{2p}}\mathrm{Wg}(N,\pi \sigma^{-1})\prod_{k=1}^{2p}\mathds{1}_{m(k)=m(\sigma(k))}
    \end{align*}
    where 
    \begin{equation}\label{m}
        m(1)=m(2)=i_1(1),\  m(3)=m(4)=i_1(2),\ \ldots,\ m(2p-1)=m(2p)=i_l(p_l).
    \end{equation}
    Since all the $i_s(t)$ are different the image set of $m$ is of cardinal $p$ and each element in this set has exactly two preimages.
    
Using  Lemma \ref{weingarten} we have 
\begin{align*}
    \sum_{\pi,\sigma \in \mathfrak{S}_{2p} }\mathrm{Wg}(N,\pi \sigma^{-1})\prod_{k=1}^{2p}\mathds{1}_{m(k)=m(\sigma(k))}&=\sum_{\sigma \in \mathfrak{S}_{2p}}(N^{-2p}+ O(N^{-2p-1}))\prod_{k=1}^{2p}\mathds{1}_{m(k)=m(\sigma(k))}\\
    &+\sum_{\substack{\sigma,\pi \in \mathfrak{S}_{2p}\\
    \pi \neq \sigma}}O(N^{-2p-1})\prod_{k=1}^{2p}\mathds{1}_{m(k)=m(\sigma(k))}\\
    &=\sum_{\sigma \in \mathfrak{S}_{2p}}N^{-2p}\prod_{k=1}^{2p}\mathds{1}_{m(k)=m(\sigma(k))} \ \ + \ \ O(N^{-2p-1})\\
    &=N^{-2p}\# \{\sigma \in \mathfrak{S}_{2p}, \ m=m\circ \sigma \} \ \ + \ \ O(N^{-2p-1}).
\end{align*}
By construction of $m$ in \eqref{m}
\begin{equation*}
    \# \{\sigma \in \mathfrak{S}_{2p}, \ m=m\circ \sigma \}=2^p.
\end{equation*}
Therefore for each $i \in T$ 
\begin{equation*}
    \E[\lvert v_{1i_1(1)}\lvert^4\ldots \lvert v_{1i_l(p_l)}\lvert^4]=N^{-2p}2^p \ \ + \ \ O(N^{-2p-1})
\end{equation*}
with the $O(N^{-2p-1})$ independent of $i$. 
Notice that $\#T=N(N-1)\ldots (N-p+1).$
Consequently
\begin{align*}
    \E[S_{n_1}^{p_1}\ldots S_{n_l}^{p_l}\overline{S_{n'_1}^{q_1}}\ldots\overline{S_{n'_r}^{q_r}} ]&=\left(\frac{N}{2}\right)^p\sum_{i\in T}(N^{-2p}2^p+O(N^{-2p-1}))p_1!\ldots p_l! \ \ + \ \ O(N^{-1})\\
    &=p_1!\ldots p_l!N^p N^{-2p}N(N-1)\ldots (N-p+1) + O(N^{-1})\\
    &\sim p_1!\ldots p_l!
\end{align*}
which concludes the proof.
\end{proof}
\bigskip
\section{Optimality Of The Critical Time}\label{Section 4}
In \cite{DubRek} the authors considered the model 
\begin{equation}\label{DubRek}
    G(t):=UA(t), \ t\in [-1,1]
\end{equation}
with $A(t)=I_N-(1-t)vv^*$ and $U$ a random unitary matrix satisfying some assumptions (which are in particular satisfied by Haar matrices and the model of Theorem \ref{extension bis}).\ $G(t)$ is a unitary matrix only for $t=\pm1$. As for $G(0)$, it is a matrix whose $N-1$ lower right minor is a truncated Haar matrix and with an added zero eigenvalue. They observed that most eigenvalues trajectories describe small loops near the boundary of the unit disk. However one trajectory crosses the disk and passes through zero when $t=0$, this is the one we call the outlier (see Figure 1 below).

They showed that the timescale $t=O(N^{-\frac{1}{2}})$ is the optimal timescale for the emergence of the outlier. To formulate more precise results, we begin by recalling some key definitions. \\
\\

We say that a sequence of event $\Omega^{(N)}$ occurs with high probability (w.h.p) if
\begin{equation*}
    \PP(\Omega^{(N)})\xrightarrow[N \to \infty]{}1.
\end{equation*}
We say that a sequence of matrix $G_N(t)$ has a strongly separated outlier towards the origin for $t \in T \subset[-1,1]$ if there exists $\alpha_1 < \alpha_2,\ d_1,d_2\in (0,1)$ such that w.h.p, for every $t\in T$, 
\begin{itemize}
    \item $D(0,d_1\frac{N^{\alpha_1}}{1+N^{\alpha_1}})$ contains exactly one eigenvalue of $G_N(t)$;
    \item $\D \backslash D(0,d_2 \frac{N^{\alpha_2}}{1+N^{\alpha_2}}) $ contains $N-1$ eigenvalues of $G_N(t).$
\end{itemize}
\begin{figure}
    \centering
    \includegraphics[width=0.3\linewidth]{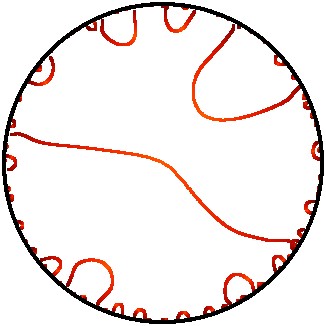}
    \caption{Trajectories of the eigenvalues of the $UA$ model of size 100 × 100. The evolution of the time parameter $t$ is represented by shades of red.}
    \label{fig:enter-label}
\end{figure}
We can now state the following result of Dubach and Reker in \cite{DubRek} (Theorem 1 and Theorem 15). 
\begin{theorem}\label{Dubach Reker}
    Let $G(t)=UA(t)$ as in \eqref{DubRek} with $U$ be an Haar distributed matrix. For any $\alpha > \eps >0$ the following holds :
    \begin{enumerate}[label=\roman*)]
        \item There is a strongly separated outlier towards the origin for $\lvert t \lvert <N^{-\frac{1}{2}-\alpha}.$
        \item With high probability, for all $\lvert t \lvert > N^{-\frac{1}{2}+\alpha},$ all eigenvalues are in $\D \backslash D(0,1-N^{-\alpha + \eps}).$
    \end{enumerate}
    Moreover, there is no strongly separated outlier at time $t=O(N^{-\frac{1}{2}}).$
\end{theorem}
In fact they showed that $i)$ and $ii)$ hold for more general matrix $U$, in particular it is also true for the model $M=VDV^{*}$ of Theorem \ref{extension bis}. However, they only proved the optimality of the critical timescale for $U$ being a Haar matrix. Indeed their proof relies crucially upon the integrability of the $UA$ model in the Haar case, which prevents it from being generalized directly. We generalize their result by adopting a different perspective. The key idea is that, at the critical timescale, the eigenvalues emerge as the zeros of a random analytic function converging to a Gaussian one which allows us to rule out the presence of an outlier. We will provide the proof for the model $M=VDV^*$, up to minor changes in notation the same argument applies for a Haar matrix $U$.\\

We now briefly outline the proof strategy.\ We aim to show that, at the critical timescale, several eigenvalues remain at an order-one distance from both the origin and the unit circle, and that there exists an order-one disk free of eigenvalues.\ Together, these observations imply that the outlier is not strongly separated at this timescale. Since the eigenvalues in this regime are the zeros of $g _N$, which converges to $\phi_{a'}$, the probability of having no eigenvalues in a disk will be approximately the same as the probability of $\phi_{a'}$ having no zeros in that disk.\ Thus, we need to estimate the probability that $\phi_{a'}$ has either no zero or multiple zeros in a disk of order one. In fact, for the proof of the optimality of the critical time scale we only need that those probabilities are strictly positive. This is the content of the next result. Recall that if $Z$ is a standard complex Gaussian random variable then $\lvert Z \lvert ^2$ follows an exponential law of parameter 1. 
\begin{lemma}\label{Fyo estimate}
  Let $q\in (0,1) $ such that $\lvert a'\lvert-\frac{q^2(2-q)}{(1-q)^2}>0$. We have that :
  \begin{enumerate}[label=\roman*)]
      \item $\PP\left(\#(D(0,q)\cap \phi_{a'}^{-1}(0))=0\right)>0$,
      \item $\PP\left(\#(D(0,q)\cap \phi_{a'}^{-1}(0))\geq 2\right)>0$.
  \end{enumerate}
\end{lemma}
\begin{proof}
  Note $r:=\frac{\lvert a'\lvert}{q}-\frac{q(2-q)}{(1-q)^2}=\frac{1}{q}\left(\lvert a'\lvert-\frac{q^2(2-q)}{(1-q)^2})\right)>0.$ Observe that $r < \frac{\lvert a' \lvert}{q}$. Let us consider the event 
  \begin{equation*}
      E:=\bigcap_{k\geq 2}\{\lvert c_k\lvert \leq k\}\cap \{\lvert c_1\lvert<r\}.
  \end{equation*}
  Note $g(z):=a'-c_1z$, $z\in \D$. On the event $E$ we have that for all $\lvert z \lvert =q$
  \begin{align*}
      \lvert \phi_{a'}(z)-g(z)\lvert &=\left \lvert -\sum_{k\geq 2}c_kz^z\right \lvert \\
      &\leq \sum_{k\geq2}kq^k\\
      &=\frac{q^2(2-q)}{(1-q)^2}.
  \end{align*}
  On $E$ we also have that for $\lvert z \lvert =q$
  \begin{align*}
      \lvert g(z) \lvert &\geq \lvert a' \lvert -\lvert c_1\lvert q
      > \lvert a' \lvert -qr
      =\frac{q^2(2-q)}{(1-q)^2}.
  \end{align*}
  Therefore, by Rouché's Theorem, on the event $E$, $\phi_{a'}$ and $g$ have the same number of zeros inside $D(0,q)$. Observe that $g(z)=0$ if and only if $c_1\neq0, \ z=\frac{a'}{c_1}$, but on the event $E$ the following inequality holds : $\lvert c_1 \lvert <r<\frac{\lvert a' \lvert}{q}$  which yields that $q<\left \lvert \frac{a'}{c_1}\right \lvert$. In conclusion, on the event $E$, $g$ does not have any zero inside $D(0,q)$ and so does $\phi_{a'}$. It only remains to prove that $\PP(E)>0$. By using the fact that the $\{c_k\}$ are independent and that $\lvert c_k \lvert ^2$ follows an exponential law of parameter $1$ we see that 
  \begin{equation*}
      \PP(E)=(1-\re^{-r^2})\prod_{k\geq 2}(1-\re^{-k^2}).
  \end{equation*}
  It is then direct that $\PP(E)>0$ and this concludes the proof of the first statement.\\
\\
  For the second statement the proof follows the same pattern. Note $h(z):=a'-c_1z-c_2z^2, \ z\in \D.$ Fix $s$ such that $s > \lv a'\lv>\frac{q^2(2-q)}{(1-q)^2}.$ Let us consider the event 
  \begin{equation*}
      B:= \bigcap_{k\geq3}\{\lvert c_k \lvert \leq k\}\cap \{\lvert c_2\lvert q^2>2s\}\cap \{s>\lvert a' \lvert + \lvert c_1 \lvert q\}.
  \end{equation*}
  On the event $B$, for all $\lvert z \lvert =q$ we have that
  \begin{align*}
      \lvert \phi_{a'}(z) -h(z)\lvert &\leq \sum_{k\geq 3}kq^k 
      \leq \sum_{k\geq2}kq^k 
      \leq \frac{q^2(2-q)}{(1-q)^2}.     
  \end{align*}
  On the event $B$ we also have that for all $\lvert z \lvert =q$
  \begin{align*}
      \lvert h(z) \lvert &\geq \lvert c_2 \lvert q^2 - \lvert a'\lvert - \lvert c_1 \lvert q 
      >2s -s
      >\frac{q^2(2-q)}{(1-q)^2}.
  \end{align*}
  Therefore by Rouché's Theorem, on the event $B$, $\phi_{a'}$ and $h$ have the same number of zeros inside of $D(0,q)$. We now show that, on $B$, $h$ has $2$ zeros in $D(0,q)$ and to do this we once again use Rouché's Theorem. On $B$ we have that for all $\lvert z \lvert =q$ 
  \begin{align*}
      \lvert h(z) - (-c_2z^2)\lvert &=\lvert a'-c_1z \lvert 
      \leq \lvert a'\lvert + \lvert c_1\lvert q 
      <s
      <\lvert -c_2z^2\lvert.
  \end{align*}
  Therefore by Rouché's Theorem, $h$ and $z\mapsto-c_2z^2$ have the same number of zeros inside $D(0,q)$, counting multiplicity, which is two. It only remains to show that $\PP(B)>0.$ By using the independence of the $\{c_k\}$ and the fact that $\lvert c_k \lvert ^2$ follows an exponential law of parameter $1$ we get that 
  
  \begin{equation*}
      \PP(B)=\re^{-\frac{s^2}{q^4}}\left(1-\re^{-\left ( \frac{s-\lvert a' \lvert}{q}\right)^2}\right)\prod_{k\geq 3}(1-\re^{-k^2}).
  \end{equation*}
  It is then direct that $\PP(B)>0$ which concludes the proof.
\end{proof}
We now state and prove our result answering a question raised in \cite{DubRek}.
\begin{theorem}
    Let $M=VDV^{*}$ as defined in Theorem \ref{extension bis} and let $G(t)=MA(t)$ as defined in \eqref{DubRek}. There is no strongly separated outlier at time $t=aN^{-\frac{1}{2}}.$
\end{theorem}
\begin{proof}
    The claim is proven by contradiction. Assume there is a strongly separated outlier at time $t= aN^{-\frac{1}{2}}.$ We proved in Theorem \ref{iid model result} that the spectrum of $G_N:=G(aN^{-\frac{1}{2}})$ is exactly the zeros of a sequence of random analytic functions $g_N$ such that $g_N \xrightarrow[N\to \infty]{d}\phi_{a'}.$ Let $q\in (0,1)$ such that $\lvert a'\lvert-\frac{q^2(2-q)}{(1-q)^2}>0$ and $q<d_2$ (the $d_2$ of the definition of strongly separated outlier just before Theorem \ref{Dubach Reker}). The function
    \begin{align*}
        H(\D&)  \rightarrow \R\\
        &f\rightarrow \inf_{\lvert z \lvert \leq q}\lvert f(z) \lvert
    \end{align*}
    is continuous. Therefore
    \begin{equation*}
        \inf_{\lvert z \lvert \leq q} \lvert g_N(z)\lvert \xrightarrow[N \to \infty]{d}\inf_{\lvert z \lvert \leq q}\lvert \phi_{a'}(z) \lvert .
    \end{equation*}
    Now note that 
    \begin{equation*}
        \PP \left(\#\left(\overline{D(0,q)}\cap\Sp(G_N)\right)=0\right)=\PP\left(\inf_{\lvert z \lvert \leq q}\lvert g_N(z)\lvert >0\right).
    \end{equation*}
    Therefore by Portmanteau's Theorem
    \begin{align*}
        \varliminf \PP\left(\inf_{\lvert z \lvert \leq q}\lvert g_N(z)\lvert >0\right)&\geq \PP\left(\inf_{\lvert z \lvert \leq q}\lvert \phi_{a'}(z)\lvert >0\right)\\
        &=\PP\left(\#\left(\overline{D(0,q)}\cap \phi_{a'}^{-1}(0)\right)=0\right).
    \end{align*}
    Lemma \ref{Fyo estimate} then yields that for all sufficiently large $N$
    \begin{equation*}
        \PP \left(\#\left (\overline{D(0,q)}\cap \Sp(G_N)\right )=0\right)\geq C_q > 0.
    \end{equation*}
    Hence, the strong separation must hold with $\alpha_1 \geq 0$ (the $\alpha_1$ of the definition of strongly separated outlier just before Theorem \ref{Dubach Reker}). Similarly we now aim to show that for $N$ large enough the probability of having two eigenvalues in $D(0,q)$ is strictly positive. 
    \\
    By Skorokhod's Theorem we can always assume that almost surely $g_N \xrightarrow[N \to \infty]{} \phi_{a'}$ in the sense of uniform convergence on compact subsets of $\D$. If we assume that $\phi_{a'}$ has at least $2$ zeros inside $D(0,q)$, then it is a direct corollary of Rouché's Theorem that for all $N$ large enough $g_N$ has at least $2$ zeros inside $D(0,q)$. Therefore we have that
    \begin{equation*}
        \{\# D(0,q)\cap \phi_{a'}^{-1}(0)\geq 2 \}\subset\bigcup_{N}\bigcap_{n\geq N}\{\#\left(D(0,q)\cap g_n^{-1}(0)\right) \geq 2\}:=E'.
    \end{equation*}
    This yields that 
    \begin{align*}
        \PP(E')&\geq \PP\left(\# D(0,q)\cap \phi_{a'}^{-1}(0)\geq 2\right)
        > C'_q 
        >0.
    \end{align*}
    Now noting that 
    \begin{equation*}
        \PP(E')=\lim_{N\to \infty}\PP \left(\bigcap_{n\geq N}\{\#\left(D(0,q)\cap g_n^{-1}(0)\right)\geq 2 \} \right )
    \end{equation*}
    we deduce that for all sufficiently large $N$ we have 
    \begin{equation*}
        \PP \left (\# \left (D(0,q)\cap \Sp(G_N)\right )\geq 2\right )\geq C_{q}'>0.
    \end{equation*}
    Hence, the strong separation must hold with $\alpha_2 \leq 0.$ This is a contradiction, as the definition of strong separation requires that $\alpha_1 < \alpha_2.$
\end{proof}


\begin{bibdiv}
\begin{biblist}
    \bib{Dja}{article}{
     title={Convergence of the spectral radius of a random matrix through its characteristic polynomial},
  author={Bordenave, C.},
  author={Chafa{\"\i}, D.},
  author={Garc{\'\i}a-Zelada, D.},
  journal={Probability Theory and Related Fields},
  pages={1--19},
  year={2022},
  publisher={Springer}
}


    \bib{DubRek}{article}{
    title={Dynamics of a rank-one multiplicative perturbation of a unitary matrix},
  author={Dubach, G.},
  author={Reker, J.},
  journal={Random Matrices: Theory and Applications},
  volume={13},
  number={02},
  pages={2450007},
  year={2024},
  publisher={World Scientific}
}
    \bib{Edmunds}{book}{
    title={Elliptic differential operators and spectral analysis},
  author={Edmunds, D. E.},
  author={Evans, W. D.},
  year={2018},
  series={Springer Monographs in Mathematics},
  publisher={Springer},
  address={New York}
    }
    \bib{ForIps}{article}{
    title={A generalisation of the relation between zeros of the complex Kac polynomial and eigenvalues of truncated unitary matrices},
  author={Forrester, P. J.},
  author={Ipsen, J. R.},
  journal={Probability Theory and Related Fields},
  volume={175},
  pages={833--847},
  year={2019},
  publisher={Springer}
}
    \bib{Fyo}{book}{
    title={Spectra of random matrices close to unitary and scattering theory for discrete-time systems},
  author={Fyodorov, Y. V.},
  book={},
  title={Spectra of random matrices close to unitary and scattering theory for discrete-time systems},
  series={in: Disordered and complex systems, AIP Conference Proceedings},
  volume={553},
  publisher={American Institute of Physics},
  address={Melville, NY},
  pages={191--196},
  year={2001}
}
   \bib{FKP}{article}{
    title={Zeros of conditional Gaussian analytic functions, random sub-unitary matrices and q-series}, 
    author={Fyodorov, Y. V.},
    author={Khoruzhenko, B. A.},
    author={Prellberg, T.},
      year={2024},
      eprint={2412.06086},
      url={https://arxiv.org/abs/2412.06086},
    }
    \bib{Sav}{book}{
    title={Resonance Scattering of Waves in Chaotic Systems},
  author={Fyodorov, Y. V.},
  author={Savin, D. V.},
  book={Resonance Scattering of Waves in Chaotic Systems},
  title={Resonance Scattering of Waves in Chaotic Systems},
  series={in: The Oxford Handbook of Random Matrix Theory (G. Akemann, J. Baik, and P.Di Francesco, eds.), Oxford Handbooks in Mathematics},
  volume={},
  publisher={Oxford University Press},
  address={Oxford},
  pages={},
  year={2011}
}
    \bib{Petz}{book}{
    title={The semicircle law, free random variables and entropy},
  author={Hiai, F.},
  author={Petz, D.},
  series={in: Mathematical Surveys and Monographs},
  number={77},
  year={2000},
  publisher={American Mathematical Society},
  address={Providence}
}
    \bib{GAF}{book}{
    title={Zeros of Gaussian analytic functions and determinantal point processes},
  author={Hough, J. B.},
  author={Krishnapur, M.},
  author={Peres, Y.},
  author={Vir{\'a}g, B.},
  series={in: University Lecture Series},
  volume={51},
  year={2009},
  publisher={American Mathematical Society},
  address={Providence}
}
    \bib{Jost}{book}{
    title={Partial differential equations},
    author={Jost, J.},
    year={2007},
    series={Graduate Texts in Mathematics},
    publisher={Springer},
    address={New York}
    }
    \bib{Kanti}{book}{
    title={Directional statistics},
    author={Jupp, P. E.},
    author={Mardia, K. V.},
  year={2009},
  publisher={John Wiley \& Sons}
  }
    \bib{Olav}{book}{
    title={Random measures, theory and applications},
  author={Kallenberg, O.},
  series={in: Probability Theory and Stochastic Modelling},
  volume={1},
  year={2017},
  publisher={Springer},
  address={New York}
}
    \bib{Krishnapur}{article}{
      title={From random matrices to random analytic functions},
  author={Krishnapur, M.},
  journal={The Annals of Probability},
  volume={37},
  number={1},
  pages={314},
  year={2009},
  publisher={Institute of Mathematical Statistics}
}
    \bib{Meckes}{book}{
      title={The random matrix theory of the classical compact groups},
  author={Meckes, E. S.},
  series={Cambridge Tracts in Mathematics},
  volume={218},
  year={2019},
  publisher={Cambridge University Press},
  address={Cambridge}
}
    \bib{Speicher}{book}{
      title={Lectures on the combinatorics of free probability},
  author={Nica, A.},
  author={Speicher, R.},
  series={in: London Mathematical Society Lecture Note Series},
  volume={13},
  year={2006},
  publisher={Cambridge University Press},
  address={Cambridge},
  
}
    \bib{Shirai}{article}{
     title={Limit theorems for random analytic functions and their zeros},
  author={Shirai, T.},
  journal={Bulletin of the American Mathematical Society},
  volume={32},
  pages={1--37},
  year={1995}
}
    \bib{Tao}{article}{
    title={Outliers in the spectrum of iid matrices with bounded rank perturbations},
  author={Tao, Terence},
  journal={Probability Theory and Related Fields},
  volume={155},
  number={1},
  pages={231--263},
  year={2013},
  publisher={Springer}
    
    }
   
\end{biblist}
\end{bibdiv}
\end{document}